\documentclass[12]{article}

\usepackage{hyperref,amsthm,amssymb,amsmath, enumerate}
\usepackage[margin=1.5in]{geometry}

\newtheorem{thm}{Theorem}[section]
\newtheorem{cor}[thm]{Corollary}
\newtheorem{lem}[thm]{Lemma}
\newtheorem{pro}[thm]{Proposition}

\theoremstyle{definition}
\newtheorem{defn}[thm]{Definition}
\newtheorem{exa}[thm]{Example}

\newcommand{\cJ}{\mathcal{J}}
\newcommand{\cM}{\mathcal{M}}
\newcommand{\cR}{\mathcal{R}}

\newcommand{\cL}{\mathcal{L}}

\newcommand{\cI}{\mathcal{I}}

\author{Casey Donoven and Luise-Charlotte Kappe}
\title{Finite Coverings of Semigroups and Related Structures}

\begin{document}
\maketitle




\begin{abstract}
For a semigroup $S$, the covering number of $S$ with respect to semigroups, $\sigma_s(S)$, is the minimum number of proper subsemigroups of $S$ whose union is $S$.  This article investigates covering numbers of semigroups and analogously defined covering numbers of inverse semigroups and monoids.  Our three main theorems give a complete description of the covering number of finite semigroups, finite inverse semigroups, and monoids (modulo groups and infinite semigroups).  For a finite semigroup that is neither monogenic nor a group, its covering number is two. For all $n\geq 2$, there exists an inverse semigroup with covering number $n$, similar to the case of loops.  Finally, a monoid that is neither a group nor a semigroup with an identity adjoined has covering number two as well.
\end{abstract}




\section{Introduction}\label{sec:introduction}
The investigations in this paper were motivated be certain results on finite coverings of groups, loops, and rings.  We say a group has a finite covering by subgroups if it is the set-theoretic union of finitely many proper subgroups.  Similarly, an algebraic structure, say a ring, a loop, or a semigroup, has a finite covering by its algebraic substructures if it is the set-theoretic union of finitely many of its proper substructures. A minimal covering for a group $G$ is a covering which has minimal cardinality amongst all the coverings of $G$.  The size of the minimal covering of a group is denoted by $\sigma(G)$.  If a group has no finite covering, we say its covering number is infinite, i.e.~$\sigma(G)=\infty$.  This group invariant was introduced in a 1994 paper by J.~H.~E.~Cohn \cite{Cohn94}, spurring a lot of research activity in this area.  However, the earliest investigations on this topic can be traced back to a 1926 paper by Scorza \cite{Scorza26}, where he proved that $\sigma(G)=3$ if and only if the Klein 4-group is a homomorphic image of $G$.

It is an easy exercise to show that no loop is the union of two proper subloops.  A simple consequence of this is that no group is the union of two proper subgroups and no ring is the union of two proper subrings.  However, it was shown by S.~Gagola III and the second author \cite{GagolaKappe16} that for every integer $n>2$, there exists a loop with covering number $n$.

The situation for groups is different.  Cohn in \cite{Cohn94} constructed a solvable group with covering number $p^\alpha+1$ for every prime $p$ and $\alpha>0$ and conjectured that every finite solvable group has a covering number of the form $p^\alpha+1$.  This was shown by Tomkinson in \cite{Tomkinson97}.  He also showed that there is no group with covering number 7 and conjectured that there are no groups with covering number 11, 13, or 15.  However, this is only true for $n=11$.  For details, see \cite{GaronziKappeSwartz18}, where it is described whether $n$ is a covering number of a group or not, for all $n$ satisfying $2\leq n\leq 129$, extending previous results from 26 to 129.

Much less is known about covering numbers of rings, but the results are similar to those concerning groups. In \cite{Lucchini12}, Lucchini and Maroti classify rings which can be covered by three proper subgroups and, in \cite{Werner15}, Werner determines the covering number of various rings which are direct sums of fields.  So far, it has not been explicitly verified if any integer $n>2$ is not a covering number of a ring.  The smallest candidate for such a number is $n=13$.

For semigroups, the topic of our investigations, the situation is completely different, as can be seen from the following example.  Consider the integers, which form a semigroup under multiplication.  Obviously, they are the union of two subsemigroups, namely the odd and even inetgers.  Semigroups having a finite covering number other than two, which are not groups, are currently being investigated in \cite{DonovenEppolito}.  It is shown that for every $n$ that is a covering number of a group with respect to groups, there exists an infinite semigroup, that is not a group, with covering number $n$ with respect to semigroups.

As we will show in our first theorem (Theorem~\ref{thm:finitesemigroups}), every finite semigroup, which is not a group or generated by a single element, has covering number two.  The following statistical evidence further illustrates the situation.  There are 1,843,120,128 non-equivalent semigroups of order eight (up to isomorphism and anti-isomorphism) \cite{SatohYamaTokizawa94}, but only 12 have covering number not equal to two.  Of the remaining 12, eight are generated by a single element and the last four are groups with semigroup covering number equal to three.

To make our notation more precise, we have to make some formal definitions.

\begin{defn}\label{defn:algebraicstructures}
\begin{enumerate}[(i)]
    \item A \emph{semigroup} is a nonempty set $S$ with an associative binary operation.
    \item A \emph{monoid} $M$ is a semigroup with an identity, i.e.~an element\linebreak $1\in M$ such that $1\cdot m=m=m\cdot 1$ for all $m\in M$.
    \item An \emph{inverse semigroup} $I$ is a semigroup such that for every element $a\in I$, there exists a unique element $a^{-1}\in I$ where $aa^{-1}a=a$ and $a^{-1}aa^{-1}=a^{-1}$.
    \item A \emph{group} $G$ is a monoid such that for every $g\in G$, there exists a unique element $g^{-1}\in G$ where $gg^{-1}=1=g^{-1}g$.
\end{enumerate}
\end{defn}

In addition to coverings of semigroups by proper subsemigroups, we also consider coverings by specific subsemigroups, such as semigroups which are groups, inverse semigroups, or monoids.  Here, we give the formal definitions of these algebraic structures and their respective covering numbers.

\begin{defn}
Let $U$ be a subsemigroup of a semigroup.
\begin{enumerate}[(i)]
    \item We say $U$ is an inverse subsemigroup of a semigroup $S$ if $U$ is an inverse semigroup.
    \item We say $U$ is a submonoid of a monoid $M$ if $U$ contains the identity of $M$.
    \item We say $U$ is a monoidal subsemigroup of a semigroup $S$ if $U$ is a monoid (but could possibly not contain the identity of the semigroup $S$, in case $S$ is a monoid).
\end{enumerate}
\end{defn}

\begin{defn}
For an algebraic structure $A$, as given in Definition~\ref{defn:algebraicstructures}, we define the following covering numbers:
\begin{enumerate}[(i)]
    \item the covering number with respect to subgroups, $\sigma_g(A)$;
    \item with respect to subsemigroups, $\sigma_s(A)$;
    \item with respect to inverse subsemigroups, $\sigma_i(A)$;
    \item with respect to submonoids, $\sigma_m(A)$;
    \item with respect to monoidal subsemigroups, $\sigma_m^*(A)$.
\end{enumerate}
\end{defn}

We are ready to state our three main results characterizing the covering numbers of finite semigroups, finite inverse semigroups, and (not necessarily finite) monoids.  The proofs are given in Section~\ref{sec:semigroups}, \ref{sec:inversesemigroups}, and \ref{sec:monoids}, respectively.

\begin{thm}\label{thm:finitesemigroups}
Let $S$ be a finite semigroup.
\begin{enumerate}[(i)]
\item If $S$ is monogenic (generated by a single element), then $\sigma_s(S)=\infty$.
\item If $S$ is a group, then $\sigma_s(S)=\sigma_g(S)$.
\item If $S$ is neither monogenic nor a group, then $\sigma_s(S)=2$.
\end{enumerate}
\end{thm}

In general, an arbitrary semigroup $S$ may not be the union of finitely many subgroups.  However, semigroups which are the union of groups have been well studied (see \cite{UnionsOfGroups}).  We present an interesting example.  

\begin{exa}
Let $S=\{a,b,0\}$ with $a\cdot a=a$, $b\cdot b=b$, and all other products equal 0.  We see $\sigma_g(S)=3$, as $\{a\}$, $\{b\}$, and $\{0\}$ are maximal subgroups of $S$.  However, $\sigma_s(S)=2$ as $S=\{a,0\}\cup\{b,0\}$.
\end{exa}

Generalizing this construction can produce semigroups with arbitrary covering numbers with respect to groups.  On the other hand, infinite groups which have no finite coverings by subgroups can have finite coverings by semigroups.

\begin{exa}
Let $\mathbb{Z}$ be the group of integers under addition.  Then $\sigma_g(\mathbb{Z})=\infty$, since $\mathbb{Z}$ is monogenic (with respect to group operations).  However, $\sigma_s(\mathbb{Z})=2$, since $\mathbb{Z}$ is the union of the positive integers and the non-positive integers.
\end{exa}

In our next theorem, we give a characterization of covering numbers of finite inverse semigroups.  Green's relations and the principal factor $J^*$ of an equivalence class $J$ are used in the statement explicitly.  For the details, we refer to Definition~\ref{defn:greensrelations} and Definition~\ref{defn:reesmatrixsemigroups} in the next section.

\begin{thm}\label{thm:inversesemigroups}
Let $I$ be a finite inverse semigroup.
\begin{enumerate}[(i)]
\item If $I$ is a group, then $\sigma_i(I)=\sigma_g(I)$.
\item If $I$ is not generated by a single $\cJ$-class, then $\sigma_i(I)=2$.
\item If $I$ is not a group but is generated by a single $\cJ$-class $J$, then $J^*$ is isomorphic to a Rees 0-matrix semigroup $\cM^0[K,G,K;P]$, where $|K|\geq 2$.  
\begin{enumerate}
    \item If $|K|=2$ and $|G|=1$, then $\sigma_i(I)=\infty$.
    \item If $|K|=2$ and $|G|>1$, then $\sigma_i(I)=n+1$ where $n$ is the minimum index of proper subgroups of $G$.
    \item If $|K|>2$, then $\sigma_i(I)=3$.
\end{enumerate}
\end{enumerate}
\end{thm}

In Proposition~\ref{pro:minimumindexsubgroup}, we show that the value of $n$ in the previous theorem can be any integer greater one except four.  Thus, we obtain the following corollary.

\begin{cor}\label{cor:allinversecoverings}
Let $n\geq 2$.  Then there exists an inverse semigroup $I$ such that $\sigma_i(I)=n$.  
\end{cor}

We prove Corollary~\ref{cor:allinversecoverings} in Section~\ref{sec:inversesemigroups} by giving examples belonging to each case of Theorem~\ref{thm:inversesemigroups}.  

Lastly, for monoids, we are able to drop the finiteness criterion.  However, this characterization is dependent on the covering number of semigroups with respect to semigroups, which is only known for finite semigroups.

\begin{thm}\label{thm:monoids}
Let $M$ be a monoid.
\begin{enumerate}[(i)]
\item If $M$ is a group, then $\sigma_m(M)=\sigma_m^*(M)=\sigma_s(M)$.
\item If $S=M-\{1\}$ is a non-empty semigroup, then $\sigma_m^*(M)\leq\sigma_m(M)=\sigma_s(S)$ and $\sigma_s(M)=2$.
\item If $M$ is neither a group nor is $M-\{1\}$ a non-empty semigroup, then\\ $\sigma_m(M)=\sigma_m^*(M)=\sigma_s(M)=2$
\end{enumerate}
\end{thm}

The second case in Theorem~\ref{thm:monoids} is the only case in which covering numbers with respect to submonoids and monoidal semigroups can differ.  We give a complete characterization of when $\sigma_m^*(M)$ is strictly less than $\sigma_m(M)$ in Section~\ref{sec:monoids}.

In our last section, Section~\ref{sec:questions}, we present some open questions concerning covering numbers of semigroups.


\section{Preliminaries}\label{sec:preliminaries}
In this section, we present various concepts and theorems needed to establish our results on covering numbers of semigroups, such as Green's relations, Rees matrix semigroups, and Rees's Theorem.  All of these definitions and results can be found in Howie's 1995 monograph \cite{howie95fundamentals}.  We will give explicit references to \cite{howie95fundamentals} but recommend this book as an excellent source for proofs and further detail.

First, we define Green's relations.  Note that we define $S^1$ as the semigroup $S$ with an identity adjoined if $S$ does not have an identity, and merely $S$ otherwise.

\begin{defn}[\cite{howie95fundamentals}; 2.1]\label{defn:greensrelations}
Let $S$ be a semigroup and $x,y\in S$.  Then 
\begin{enumerate}[(i)]
    \item $x\cJ y$ if and only if $S^1xS^1=S^1yS^1$;
    \item $x\cR y$ if and only if $xS^1=yS^1$;
    \item $x\cL y$ if and only if $S^1x=S^1y$.    
\end{enumerate}
\end{defn}

It can be easily seen that $\cJ$, $\cR$, and $\cL$ are equivalence relations.  A useful equivalent definition is given in the following proposition.

\begin{pro}[\cite{howie95fundamentals},2.1.1]
Let $x,y$ be elements in a semigroup $S$.  Then $x\cJ y$ if and only if there exist $a,b,c,d\in S^1$ such that $axb=y$ and $cyd=x$.  Similarly, $x\cR y$ if and only if there exist $a,b\in S$ such that $xa=y$ and $yb=x$, as well as $x\cL y$ if and only if there exist $c,d\in S$ such that $cx=b$ and $dy=x$.
\end{pro}

Equivalence classes under the $\cJ$, $\cR$, and $\cL$ relations are called $\cJ$-classes, $\cR$-classes, and $\cL$-classes, respectively.  Also, the $\cJ$-class, $\cR$-class, and $\cL$-class containing the element $x$ is denoted by $J_x$, $R_x$, and $L_x$. 

There is a natural partial order, $\leq_{\cJ}$, on the $\cJ$-classes of $S$ where, for $x,y\in S$, we have $J_x\leq_{\cJ} J_y$ if and only if $S^1xS^1\subseteq S^1yS^1$.  Similar partial orders on $\cR$- and $\cL$-classes exist, where for $x,y\in S$, $R_x\leq_{\cR} R_y$ if and only if $xS^1\subseteq yS^1$ and $L_x\leq_{\cL} L_y$ if and only if $S^1x\subseteq S^1y$.

The following is a useful result describing where products of elements lie in the partial order on Green's classes.

\begin{lem}[\cite{howie95fundamentals}, 2.1]\label{lem:Jpartialorder}
For all $x,y\in S$, $J_{xy}\leq_\cJ J_x$, $J_{xy}\leq_\cJ J_y$, $R_{xy}\leq_\cR R_x$, and\linebreak $L_{xy}\leq_\cL L_{y}$.
\end{lem}

Our investigations of covering numbers use several classification results for semigroups, namely Rees matrix semigroups and Rees 0-matrix semigroups, which we describe now.

\begin{defn}\label{defn:reesmatrixsemigroups}
Let $K$ and $\Lambda$ be nonempty sets and let $G$ be a group.
\begin{enumerate}[(i)]
    \item Let $P$ be a $|\Lambda|\times |K|$ matrix with entries in $G$.  Then the \emph{Rees matrix semigroup} $S=\cM[K,G,\Lambda;P]$ is the set of triples $K\times G\times \Lambda$ with multiplication defined by $$(\kappa,g,\lambda)(\mu,h,\nu)=(\kappa,gp_{\lambda,\mu}h,\nu).$$
    \item Let $Q$ be a $|\Lambda|\times |K|$ matrix over $G\cup \{0\}$. Then the \emph{Rees 0-matrix semigroup}\linebreak $S=\cM^0[K,G,\Lambda;Q]$ is the set $(K\times G\times \Lambda)\cup \{0\}$ with multiplication defined by $$(\kappa,g,\lambda)(\mu,h,\nu)=(\kappa,gq_{\lambda,\mu}h,\nu)$$ when $q_{\lambda,\mu}\neq 0$, $$(\kappa,g,\lambda)(\mu,h,\nu)=0$$ when $q_{\lambda,\mu}= 0$, and $$0\cdot s=s\cdot0=0$$ for all $s\in S$.  The matrix $Q$ is called \emph{regular} if each row and column contains a non-zero element.
\end{enumerate}
\end{defn}

Rees's Theorem \cite{Rees40} characterizes semigroups with certain $\cJ$-class structure.  To state the theorem, we need a few more definitions.

\begin{defn}[\cite{howie95fundamentals}, 3.2]
\begin{enumerate}[(i)]
    \item A semigroup $S$ is \emph{simple} if $S$ is comprised of a single\linebreak $\cJ$-class. 
    \item A semigroup $S$ with a zero, i.e.~$0\cdot s=0=s\cdot 0$ for all $s\in S$, is \emph{0-simple} if $S$ is a semigroup with two $\cJ$-classes, where one is the set $\{0\}$ and $S^2\neq \{0\}$.
    \item A semigroup $S$ is \emph{completely simple} or \emph{completely 0-simple} if $S$ is simple or 0-simple, respectively, and every non-empty set of $\cR$-classes and every non-empty set of $\cL$-classes has a minimal element.
\end{enumerate}
\end{defn}

Before stating Rees's Theorem, we note that finite simple and 0-simple semigroups are completely simple and completely 0-simple.

\begin{thm}[\cite{howie95fundamentals}, Theorem 3.2.3]\label{thm:Rees}
A semigroup $S$ is completely simple if and only if $S$ is isomorphic to a Rees matrix semigroup.  Also, $S$ is completely 0-simple if and only if $S$ is isomorphic to a Rees 0-matrix semigroup with a regular matrix.
\end{thm}

The final construction we have to mention is the principle factor.  It is defined as follows.

\begin{defn}[\cite{howie95fundamentals}, 3.1]\label{defn:principalfactor}
Let $S$ be a semigroup and $J$ be a $\cJ$-class of $S$.  The \emph{principle factor} of $\cJ$, denoted by $J^*$, of $J$ is a semigroup with elements $J\cup \{0\}$ and operation $*$ such that for $s,t\in J^*$, we have $s*t=st$ when $s,t,st\in J$ and $s*t=0$ otherwise. 
\end{defn}

Essentially, products in $J^*$ are the same as they are in $J$, but are set equal to 0 when the product lies outside of $J$.  Furthermore, the following property of $J^*$ is of interest in our investigations.  When $J$ is a maximal $\cJ$-class of a semigroup $S$ that is not simple, there is a natural surjective homomorphism $\phi : S \to J^*$ where $(s)\phi = s$ when $s \in J$ and $(s)\phi = 0$ otherwise

We conclude our list of preparatory results with a theorem characterizing principal factors. We note that a \emph{null} semigroup is a semigroup with a 0 such that every product is 0.

\begin{thm}[\cite{howie95fundamentals}, 3.1.6]\label{thm:Jstar}
Let $S$ be a semigroup and $J$ be a $\cJ$-class of $S$. Then $J^*$ is 0-simple or null.
\end{thm}


\section{Covering finite semigroups}\label{sec:semigroups}
In this section, we give a proof of Theorem~\ref{thm:finitesemigroups}.  First we prove various lemmas, before presenting a cohesive proof at the end of the section.

We begin with the following observation about torsion groups.
\begin{lem}\label{lem:torsion}
If $S$ is a torsion group, then $\sigma_s(S)=\sigma_g(S)$.
\end{lem}
\begin{proof}
Clearly, every subgroup of $S$ is also a subsemigroup.  Let $T$ be a subsemigroup of $S$ and let $x\in T$.  Since $S$ is torsion, there exists an $n\in\mathbb{N}$ such that $x^n=1$ and $x^{n-1}=x^{-1}$.  Note that $T$ is closed under multiplication and therefore $T$ contains the identity and $x^{-1}$. We see $T$ is a group.  Thus subsemigroups of $S$ are also subgroups of $S$, and we conclude $\sigma_s(S)=\sigma_g(G)$.
\end{proof}

The following corollary is an immediate consequence. 

\begin{cor}\label{cor:finitegroups}
If $S$ is a finite group, then $\sigma_s(S)=\sigma_g(S)$
\end{cor}

The $\cJ$-class structure of semigroups allows us to find proper subsemigroups. The methods used to construct these proper subsemigroups were inspired by a 1968 paper by Graham et al.~\cite{GrahamGrahamRhodes68}, in which the maximal proper subsemigroups of an arbitrary finite semigroup are characterized.

\begin{lem}\label{lem:SwithoutJ}
Let $S$ be a semigroup and let $J$ be a maximal $\cJ$-class of $S$ under the partial order. Then the set difference $S-J$ is a semigroup, provided $S-J\neq \emptyset$.
\end{lem}
\begin{proof}
Let $x,y\in S-J$. We will show $xy\in S-J$.  Either $J_x$ is incomparable to $J$ or $J_x<_{\cJ} J$, since $x\not\in J$.  By Lemma~\ref{lem:Jpartialorder}, we have $J_{xy}\leq_{\cJ} J_x$. With transitivity of the partial order, we see $J_{xy}<J$ or $J_{xy}$ is incomparable to $J$. We conclude that $xy\not\in J$ and thus $xy\in S-J$.
\end{proof}

\begin{cor}\label{cor:semigroupwithnogenJ}
Let $S$ be a semigroup with a maximal $\cJ$-class $J$ such that $\langle J \rangle\neq S$.  Then $\sigma_s(S)=2$.
\end{cor}
\begin{proof}
We have $S=(S-J)\cup\langle J \rangle$ and if $\langle J \rangle\neq S$, then $S-J$ is non-empty.
\end{proof}

Now consider a finite semigroup $S$. We see that $S$ will have at least one maximal $\cJ$-class, $J$. Corollary~\ref{cor:semigroupwithnogenJ} says that $\sigma_s(S)=2$ unless $\langle J \rangle=S$.  This leaves two cases: when $J=S$ and when $\langle J \rangle=S$ but $J\neq S$.

Beginning with the case when $J=S$, recall that Rees's Theorem states that when $S$ is a finite semigroup with a single $\cJ$-class, $S$ is isomorphic to a Rees matrix semigroup.

\begin{lem}\label{lem:reesmatrix}
Let $S=\cM[K,G,\Lambda;P]$ be a Rees matrix semigroup. If $|K|>1$ or $|\Lambda|>1$, then $\sigma_s(S)=2$.  If $|K|=1$ and $|\Lambda|=1$, then $S$ is a group.
\end{lem}
\begin{proof}
First, consider the case when $|K|>1$.  Let $\kappa\in K$ and consider the subset\linebreak $T=\{\kappa\}\times G\times \Lambda$ of $S$.  For $(\kappa,g,\lambda),(\kappa,h,\mu)\in T$, we have $$(\kappa,g,\lambda)(\kappa,h,\mu)=(\kappa,gp_{\lambda,\kappa}h,\mu)\in T,$$ so $T$ is a proper subsemigroup of $S$.  Next, consider the complement of $T$ in $S$, i.e.~the set $S-T$, where $u\in S-T$ whenever $u=(\nu,g,\lambda)$ with $\nu\neq \kappa$.  Obviously, for $u,u'\in S-T$, we have $uu'\in S-T$.  Thus $S-T$ is a subsemigroup of $S$.  We conclude $S=T\cup (S-T)$ and $\sigma_s(S)=2$.

The case when $|\Lambda|>1$ is handled similarly.

Lastly, we consider the case when $|K|=|\Lambda|=1$. Let $S=\{\kappa\}\times G \times \{\lambda\}$ and $P=[g]$ where $g\in G$.  Through direct calculation, we see that the element $(\kappa,g^{-1},\lambda)$ is an identity in $S$ and the element $(\kappa,g^{-1}h^{-1}g^{-1},\lambda)$ is the inverse of $(\kappa,h,\lambda)$.  Therefore $S$ is a group.
\end{proof}

We now consider the second case, where the semigroup $S$ has a maximal $\cJ$-class $J$ that generates $S$ but $J \neq S$. Let $J^*$ be the principal factor of $J$ (see Definition~\ref{defn:principalfactor}).  Recall that we have a surjection $S\to J^*$ and that the principal factor $J^*$ is either null or 0-simple by Theorem~\ref{thm:Jstar}.  Furthermore, applying Theorem~\ref{thm:Rees}, we obtain that $J^*$ is null or isomorphic to a Rees 0-matrix semigroup with a regular matrix. We first consider the case that $S$ is a Rees 0-matrix semigroup.

\begin{lem}
Let $S=\cM^0[K,G,\Lambda;P]$ be a Rees 0-matrix semigroup with a regular matrix $P$. Then $\sigma_s(S)=2$.
\end{lem}
\begin{proof}
If $|K|>1$, then let $\kappa\in K$ and we see $R=(\{\kappa\}\times G\times \Lambda)\cup \{0\}$ is a proper subsemigroup of $S$, using a similar argument as in the proof of Lemma~\ref{lem:reesmatrix}.  Similiarly, we see $(S-R)\cup \{0\}$ is another proper subsemigroup of $S$. Thus $S=R\cup ((S-R)\cup \{0\})$ and hence $\sigma_s(S)=2$.  When $|\Lambda|>1$, the same technique applies.

Now consider the case when $|K|=|\Lambda|=1$.  The single entry in the matrix $P$ must be an element of $G$ and therefore $K\times G\times \Lambda$ is a proper subsemigroup of $S$.  Since $\{0\}$ is also a proper subsemigroup of $S$, we have $\sigma_s(S)=2$
\end{proof}

By taking preimages using the surjection $\phi:S\to J^*$, the following is clear.

\begin{cor}\label{cor:rees0matrixsemigroups}
Let $S$ be a semigroup with a $\cJ$-class $J$ such that $S=\langle J\rangle$.  If $S\neq J$ and $J^*$ is isomorphic to a regular Rees 0-matrix semigroup, then $\sigma_s(S)=2$.
\end{cor}

It remains to consider the case when $J^*$ is null.

\begin{lem}\label{lem:nullismonogenic}
Let $S$ be a semigroup with a $\cJ$-class $J$ such that $S=\langle J\rangle$.  If $S\neq J$ and $J^*$ is null, then $S$ is monogenic and $\sigma_s(S)=\infty$.
\end{lem}
\begin{proof}
We will show that $|J|=1$.  Let $x,y\in J$.  Then there exist elements $a,b\in S^1$ such that $axb=y$.  Assume for contradiction that $a\neq 1$ or $b\neq 1$.  Since $J$ generates $S$, at least one of $a$ or $b$ is a product of elements in $J$.  However, $J^*$ is null, meaning that the product of elements from $J$ is not contained in $J$, i.e.~$a\not\in J$ or $b\not\in J$.  This shows that $axb\not\in J$ and thus $axb\neq y$.  This is a contradiction, and therefore $a=1$ and $b=1$.  We conclude that $x=y$ and $|J|=1$.  Therefore $S$ is monogenic and $\sigma_s(S)=\infty$.
\end{proof}

We now present the proof of Theorem~\ref{thm:finitesemigroups}, using the above lemmas and corollaries.

\begin{proof}[Proof of Theorem~\ref{thm:finitesemigroups}]
Let $S$ be a finite semigroup.  If $S$ is not generated by a single $\cJ$-class, then $\sigma_s(S)=2$ by Corollary~\ref{cor:semigroupwithnogenJ}.  If $S$ is generated by a single $\cJ$-class $J$, there are two cases to consider: when $S=J$ and when $S\neq J$.

In the case that $S=J$, we have $S$ is a Rees matrix semigroup.  Using Corollary~\ref{cor:finitegroups} and Lemma~\ref{lem:reesmatrix}, we see that either $S$ is a group and $\sigma_s(S)=\sigma_g(S)$, or otherwise $\sigma_s(S)=2$.

Lastly, in the case that $S\neq J$, then $S$ surjects onto $J^*$, which is either a Rees 0-matrix semigroup or a null semigroup. Corollary~\ref{cor:rees0matrixsemigroups} and Lemma \ref{lem:nullismonogenic} imply that $\sigma_s(S)=2$ when $J^*$ is a Rees 0-matrix semigroup or $\sigma_s(S)=\infty$ when $J^*$ is null, since $S$ is monogenic.
\end{proof}


\section{Covering Finite Inverse Semigroups}\label{sec:inversesemigroups}
In this section, we give a proof of Theorem~\ref{thm:inversesemigroups}, which deals with covering numbers of finite inverse semigroups, as given in Definition~\ref{defn:algebraicstructures}.  Several important facts about inverse semigroups are summarized in the following lemma.  For further details, we refer to Chapter 5 of \cite{howie95fundamentals}.

\begin{lem}\label{lem:inversesemigroupbasics}
Let $I$ be an inverse semigroup.  Then $a\cJ a^{-1}$, $(a^{-1})^{-1}=a$, and\linebreak $(ab)^{-1}=b^{-1}a^{-1}$ for all $a,b\in I$.  
\end{lem}

Our proof of Theorem~\ref{thm:monoids} splits into three cases: when $I$ is a group, when $I$ is not generated by a single $\cJ$-class, and otherwise.  The first two cases are very easy and follow along the line of the proofs in Section~\ref{sec:semigroups}. By Lemma~\ref{lem:torsion}, the case when a finite inverse semigroup is a group is clear.  In the case that $I$ is not generated by a single $\cJ$-class, the same technique of Lemma~\ref{lem:SwithoutJ} applies to inverse semigroups, taking the set difference with a maximal $\cJ$-class.

\begin{lem}\label{lem:inverseSwithoutJ}
Let $I$ be a finite inverse semigroup and $J$ be a maximal $\cJ$-class of $I$.  Then $I-J$ is an inverse subsemigroup of $I$.
\end{lem}
\begin{proof}
By Lemma~\ref{lem:SwithoutJ}, it follows that $I-J$ is a subsemigroup of $I$.  Let $a\in I-J$.  By Lemma~\ref{lem:inversesemigroupbasics}, we have $a^{-1}\cJ a$ and so $a^{-1}\in I-J$.  Therefore $I-J$ is an inverse subsemigroup.
\end{proof}

Similar to Corollary~\ref{cor:semigroupwithnogenJ} for finite semigroups, we have an analogue for finite inverse semigroups.

\begin{cor}\label{cor:inversesemigroupwithoutJ}
Let $I$ be a finite inverse semigroup and $J$ be a maximal $\cJ$-class of $I$.  If $\langle J\rangle\neq I$, then $\sigma_i(I)=2$.
\end{cor}
\begin{proof}
By Lemma~\ref{lem:inverseSwithoutJ}, we have that $I-J$ is an inverse subsemigroup of $I$.  Also, $\langle J\rangle$ is an inverse subsemigroup, because $(ab)^{-1}\in \langle J\rangle$ by Lemma~\ref{lem:inversesemigroupbasics}.  Since $\langle J\rangle\neq I$, we have $I=(I-J)\cup \langle J\rangle$ and hence $\sigma_i(I)=2$.
\end{proof}

We now consider when $I$ is generated by a single $\cJ$-class $J$.  By Rees's Theorem, if $I=J$, then $I$ is a Rees matrix semigroup and $J^*$ is a Rees 0-matrix semigroup.  The following theorem gives insight into the structure of Green's equivalence classes in inverse semigroups, including inverse Rees matrix and 0-matrix semigroups.

\begin{thm}[\cite{howie95fundamentals}, 5.1.1]\label{thm:inverseidempotents}
For a semigroup $S$, the following are equivalent:
\begin{enumerate}[(i)]
    \item $S$ is an inverse semigroup;
    \item every $\cR$-class as well as every $\cL$-class contains exactly one idempotent.
\end{enumerate}
\end{thm}

Since each $\cR$-class as well as $\cL$-class of an inverse semigroup contains a unique idempotent, there must be the same number of  $\cR$- and $\cL$-classes.   We use this fact to prove the following lemma.

\begin{lem}\label{lem:inversereesmatrix}
Let $I$ be a finite inverse Rees matrix semigroup.  Then $I$ is a group and $\sigma_i(I)=\sigma_g(I)$.
\end{lem}

\begin{proof}
Let $I=\cM[K,G,\Lambda;P]$ be an inverse semigroup.  For any $\kappa\in K$ and any $\lambda\in\Lambda$, we have $x=(\kappa,p_{\lambda,\kappa},\lambda)$ is an idempotent with $$R_x=\{(\kappa,g,\mu)\mid g\in G,\mu\in\Lambda\}\quad\mathrm{and}\quad L_x=\{(\nu,g,\lambda)\mid g\in G,\nu\in K\}.$$  If $|K|>1$ and $\kappa\neq \nu$, then $(\kappa,p_{\lambda \kappa}^{-1},\lambda)$ and $(\nu,p_{\lambda \nu}^{-1},\lambda)$ would be two distinct idempotents in $L_x$, a contradiction.  Thus $|K|=1$ and similarly $|\Lambda|=1$.  We conclude $\cM[K,G,\Lambda;P]\cong G$ and $\sigma_i(I)=\sigma_g(I)$.
\end{proof}

For an inverse Rees 0-matrix semigroup $I=\cM^0[K,G,\Lambda;P]$, Theorem~\ref{thm:inverseidempotents} implies\linebreak $|K|=|\Lambda|$ as elements of $K$ and $\Lambda$ correspond to distinct $\cR$- and $\cL$-classes, respectively.  We will simply assume $K=\Lambda$ for ease of notation. Furthermore, if $p_{\lambda,\kappa}\neq 0$, then $(\kappa,p_{\lambda,\kappa}^{-1},\lambda)$ is an idempotent, so $P$ contains a single non-zero entry in each row and column. 

Let $e\in G$ be the identity of $G$.  Without loss of generality (using Theorem~3.4.3 from \cite{howie95fundamentals}), we will assume $p_{\kappa,\kappa}=e$ for all $\kappa\in K$ and $p_{\lambda,\kappa}=0$ otherwise. Notice that this is just a reordering of the rows and columns of the matrix $P$ so it is diagonal, and a \emph{normalization} of entries to be equal to the identity or 0.  Essentially, $P$ is the identity matrix.

The fact that the matrix of an inverse Rees 0-matrix semigroup is diagonal helps lead to the following lemma. Recall that $J^*$ is the principle factor of a $\cJ$-class $J$, given in Definition~\ref{defn:principalfactor}

\begin{lem}\label{lem:inverseress0matrix}
Let $I$ be a finite inverse semigroup with maximal $\cJ$-class $J$ such that $\langle J\rangle = I$.  Then $J^*$ is isomorphic to a Rees 0-matrix semigroup $\cM^0[K,G,K;P]$ such that $|K|\geq2$.
\end{lem}
\begin{proof}
By Theorem~\ref{thm:Rees} and the paragraphs preceding this lemma, $J^*$ is null or is isomorphic to a Rees 0-matrix semigroup $\cM^0[K,G,\Lambda;P]$ with identity matrix $P$. However, $J^*$ cannot be null as only 0 would have an inverse. Also, if $|K|=1$, then $K\times G\times K$ would be a group and $\langle J\rangle\neq I$.  Therefore $|K|>1$.
\end{proof}

In this case, where $I$ is generated by a single $\cJ$-class but is not equal to a single $\cJ$-class, the following lemma shows that $I$ is not the union of two proper inverse subsemigroups.

\begin{pro}\label{pro:inversereesnotcoveredby2}
Let $I=\cM^0[K,G,K;P]$ be a finite inverse Rees 0-matrix semigroup, where $|K|\geq2$.  Then $\sigma_i(I)\neq 2$.
\end{pro}
\begin{proof}
Let $H$ be a proper inverse subsemigroup of $I=\cM[K,G,K;P]$, where $|K|\geq2$, and $H^c$ be the complement of $H$ in $I$.  We will prove that $\langle H^c\rangle=I$ in three cases:
\begin{enumerate}[(i)]
\item $H$ does not contain every non-zero idempotent of $I$;
\item $H$ contains every non-zero idempotent of $I$ and not every non-zero idempotent is\linebreak $\cJ^H$-related;
\item $H$ contains every non-zero idempotent of $I$ and each non-zero idempotent is\linebreak $\cJ^H$-related.
\end{enumerate}
Note that each non-zero idempotent of $I$ is of the form $(\kappa,e,\kappa)$ for some $\kappa\in K$.

\begin{enumerate}[(i)]
    \item Let $\kappa\in K$ such that $(\kappa,e,\kappa)\not\in H$.  Also let $\lambda,\mu\in K$ and $g\in G$ so that $(\lambda,g,\mu)\in I$.  We first show that $(\kappa,g,\mu)\not\in H$ via contradiction.  Suppose $(\kappa,g,\mu)\in H$.  Since $H$ is an inverse semigroup, we have $(\kappa,g,\mu)^{-1}\in H$ with $(\kappa,g,\mu)^{-1}=(\mu,g^{-1},\kappa)$.  Furthermore, $$(\kappa,g,\mu)(\mu,g^{-1},\kappa)=(\kappa,e,\kappa)$$ which contradicts the fact that $H$ is closed.  By a similar argument, $(\lambda,e,\kappa)\not\in H$.  Therefore $$(\lambda,e,\kappa)(\kappa,g,\mu)=(\lambda,g,\mu)\in \langle H^c\rangle,$$  and hence, we have $\langle H^c\rangle=I$.
    \item Let $\lambda,\mu\in K$. First, we consider the subcase when $(\lambda,e,\lambda)$ and $(\mu,e,\mu)$ are not\linebreak $\cJ^H$-related.  We claim this implies $(\lambda,g,\mu)\not\in H$ for each $g\in G$.  Suppose to the contrary that $(\lambda,g,\mu)\in H$.  Since $H$ is an inverse semigroup, we have $(\lambda,g,\mu)^{-1}\in H$ with  $(\lambda,g,\mu)^{-1}=(\mu,g^{-1},\lambda)$.  Furthermore, $$(\lambda,g,\mu)(\mu,e,\mu)(\mu,g^{-1},\lambda)=(\lambda,e,\lambda)$$ and $$(\mu,g^{-1},\lambda)(\lambda,e,\lambda)(\lambda,g,\mu)=(\mu,e,\mu),$$ which contradicts the fact that $(\lambda,e,\lambda)$ is not $\cJ^H$-related to $(\mu,e,\mu)$.  Therefore $(\lambda,g,\mu)\in H^c$.
    
    Now suppose that $(\lambda,e,\lambda)$ is $\cJ^H$-related to $(\mu,e,\mu)$.  By assumption, this implies that there exists a $\kappa\in K$ such that $(\kappa,e,\kappa)$ is neither $\cJ^H$-related to $(\lambda,e,\lambda)$ nor to $(\mu,e,\mu)$.  By the previous arguments, this implies $(\lambda,g,\kappa)\in H^c$ and $(\kappa,e,\mu)\in H^c$.  Therefore $$(\lambda,g,\kappa)(\kappa,e,\mu)=(\lambda,g,\mu)\in \langle H^c\rangle.$$ We have shown $\langle H^c\rangle=I$ in this case.
    \item Let $\kappa,\lambda,\mu\in K$.  This means that $(\kappa,e,\kappa)$, $(\lambda,e,\lambda)$, and $(\mu,e,\mu)$ are idempotents contained in $H$.  By our assumption, there exist elements $(\kappa,g,\lambda),(\kappa,h,\mu)\in H$ such that $$(\kappa,g,\lambda)^{-1}(\kappa,e,\kappa)(\kappa,g,\lambda)=(\lambda,e,\lambda)$$ and $$(\kappa,h,\mu)^{-1}(\kappa,e,\kappa)(\kappa,h,\mu)=(\mu,e,\mu).$$  Let $H_{\kappa,\lambda}=H\cap (\{\lambda\}\times G\times \{\lambda\})$.  We see that $H_{\lambda,\mu}=(\mu,g,\lambda)^{-1}H_{\kappa,\kappa}(\kappa,h,\mu)$ and therefore we have $|H_{\kappa,\kappa}|=|H_{\lambda,\mu}|$.  Since $H\neq I$, it follows that $H_{\lambda,\mu}\subsetneq\{\lambda\}\times G\times \{\mu\}$. Therefore $G$ is not the trivial group.  We can also see that $\{\kappa\}\times G\times \{\kappa\}$  is a group, of which $H_{\kappa,\kappa}$ is a subsemigroup. Lemma~\ref{lem:torsion} then implies that $H_{\kappa,\kappa}$ is a group.  Since $\{\kappa\}\times G\times \{\kappa\}$ is a group and no group is the union of two proper subgroups, we conclude that $$\langle H^c\cap (\{\kappa\}\times G\times \{\kappa\})\rangle=\{\kappa\}\times G\times \{\kappa\}.$$  However, there exist elements $(\kappa,g',\lambda),(\kappa,h',\mu)\in H^c$, since $$H_{\kappa,\lambda}\subsetneq\{\kappa\}\times G\times \{\lambda\}\quad\mathrm{and}\quad H_{\kappa,\mu}\subsetneq\{\kappa\}\times G\times \{\mu\}.$$  Therefore $\langle H^c\cap (\{\lambda\}\times G\times \{\mu\})\rangle=\{\lambda\}\times G\times \{\mu\}$.  We conclude $\langle H^c\rangle=I$.
\end{enumerate}
\end{proof}

We now separate the size of the index set $K$ into two cases, namely when $|K|\geq 3$ and when $|K|=2$.  We first consider the case $|K|\geq 3$, which is much simpler.

\begin{pro}\label{pro:inversereesKnot2}
Let $I=\cM^0[K,G,K;P]$ be a finite inverse Rees 0-matrix semigroup, where $|K|\geq 3$.  Then $\sigma_i(I)= 3$.
\end{pro}
\begin{proof}
Let $\kappa_1, \kappa_2, \kappa_3$ be distinct elements in $K$.  Define $$H_j=((K-\{\kappa_j\})\times G\times (K-\{\kappa_j\}))\cup \{0\}$$ for $j=1,2,3$.  It is clear that $H_j$ is a subsemigroup of $I$, since no product of elements in $H_j$ will contain $\kappa_j$ in its tuple.  Also, $H_j$ is an inverse subsemigroup, since the inverse of an element without $\kappa_j$ in its tuple also does not have $\kappa_j$ in its tuple.  Therefore $\sigma_i(I)=3$, because $I=\bigcup H_j$ and $\sigma_i(I)\neq 2$ by Proposition~\ref{pro:inversereesnotcoveredby2}.
\end{proof}

\begin{pro}\label{pro:inversereesKis2}
Let $I=\cM^0[K,G,K;P]$ be a finite inverse Rees 0-matrix semigroup where $|K|=2$.  If $|G|=1$, then $\sigma_i(I)=\infty$.  Otherwise, if $|G|>1$, then $\sigma_i(I)=n+1$, where $n$ is the minimum index of proper subgroups of $G$.
\end{pro}
\begin{proof}
Let $K=\{1,2\}$ and $I=\cM^0[K,G,K;P]$.  We first consider the case when $|G|=1$.  Let $x=(1, e, 2)$. Then $x^{-1}=(2, e, 1)$,  $xx^{-1}=(1,e,1))$, $x^{-1}x=(2,e,2)$, and $xx=0$.  Therefore $I$ is monogenic, with respect to inverse semigroup operations, and $\sigma_i(I)=\infty$.  

We now consider the case when $|G|>1$.  Let $n$ be the minimum index of proper subgroups of $G$.  Also let $\{H_1,\ldots,H_m\}$ be a covering of $I$ by $m$ proper inverse subsemigroups.  Consider $$T_i^{\kappa,\lambda}=H_i\cap (\{\kappa\}\times G\times \{\lambda\})$$ for each $i\leq m$ and $\kappa,\lambda\in K$. 

Without loss of generality, assume $T_i^{1,2}\neq \emptyset$.  We first show $|T_i^{1,2}|\leq |G|/n$, before describing the elements of $T_i^{1,2}$ more explicitly.  We see that $T_i^{1,2}T_i^{2,1}\subseteq T_i^{1,1}$ and\linebreak $T_i^{1,1}T_i^{1,2}\subseteq T_i^{1,2}$.  This implies $|T_i^{1,1}|=|T_i^{1,2}|$.  If $T_i^{1,2}=\{1\}\times G\times \{2\}$, then $|T_i^{\kappa,\lambda}|=|G|$ for each $j$ and $k$. This is a contradiction as $H_i$ is a proper inverse subsemigroup, and thus $|T_i^{1,2}|<|G|$.  Furthermore, $T_i^{1,1}$ is a subgroup of $\{1\}\times G\times \{1\}$.  Since $\{1\}\times G\times \{1\}$ is isomorphic to $G$ and $|T_i^{1,2}|<|G|$, we have $|T_i^{1,2}|\leq |G|/n$.  Here, we may immediately conclude $m\geq n$, since $\{H_i,\ldots,H_m\}$ must cover $\{1\}\times G\times \{2\}$.

Now suppose $T_i^{1,2}=\{1\}\times A_i\times \{2\}$.  Since $T_i^{1,1}$ is a subgroup of $\{1\}\times G\times \{1\}$, we have $T_i^{1,1}=\{1\}\times B_i\times \{1\}$ for some subgroup $B_i$ of $G$.  As above, $T_i^{1,1}T_i^{1,2}=T_i^{1,2}$, implying  $B_iA_i=A_i$ and $|B_i|=|A_i|$.  Therefore $A_i$ must be a coset of $B_i$. We may now conclude $m>n$, since there do not exist $n$ proper subgroups of $G$ that cover $G$ but $\{H_1,\ldots, H_m\}$ must cover $\{1\}\times G\times \{1\}$.

Finally, we give a cover of $I$ using $n+1$ proper inverse subsemigroups.  Let $B$ be a subgroup of $G$ of index $n$, and let $g_1,\ldots,g_n$ be coset representatives of $B$.  For $i\leq n$ define 
\begin{align*}
    H_i=(\{1\}\times B\times \{1\})&\cup (\{1\}\times Bg_i\times \{2\})\cup (\{2\}\times g_i^{-1}B\times \{2\})\\&\cup (\{2\}\times g_i^{-1}Bg_i\times \{2\})\cup \{0\}.    
\end{align*}
It is routine to check that $H_i$ is an inverse subsemigroup of $I$.  Also define $$H_{n+1}= (\{1\}\times G\times \{1\})\cup \{2\}\times G\times \{2\}\cup \{0\}.$$
Similarly, $H_{n+1}$ is an inverse subsemigroup of $I$ and $\{H_1,\ldots, H_{n+1}\}$ forms a covering of $I$.  We conclude $\sigma_i(I)=n+1$.
\end{proof}

We now present the proof of Theorem~\ref{thm:inversesemigroups}.

\begin{proof}[Proof of Theorem~\ref{thm:inversesemigroups}]
Let $I$ be a finite inverse semigroup.  If $I$ is not generated by a single $\cJ$-class, then $\sigma_i(I)=2$ by Corollary~\ref{cor:inversesemigroupwithoutJ}.  If $I$ is generated by a single $\cJ$-class $J$, then there are two cases to consider: when $J=I$ and when $J\neq I$.

When $J=I$, we see that $I$ is a Rees matrix semigroup using Theorem~\ref{thm:Rees}.    Lemma~\ref{lem:inversereesmatrix} then implies $I$ is a group with $\sigma_i(I)=\sigma_g(I)$.

When $J\neq I$ but $\langle J\rangle =I$,  we see $I$ surjects onto $J^*$ and $J^*$ is isomorphic to a Rees 0-matrix semigroup $\cM^0[K,G,K;P]$, where $|K|\geq 2$, by Lemma~\ref{lem:inverseress0matrix}.  Using Propositions~\ref{pro:inversereesnotcoveredby2} and \ref{pro:inversereesKnot2}, we see that when $|K|\geq 3$, we have $\sigma_i(J^*)=3$. Furthermore, with Proposition~\ref{pro:inversereesKis2}, when $|K|=2$, either $\sigma_i(J^*)=\infty$, when $|G|=1$, or $n+1$ otherwise, where $n$ is the minimum index of proper subgroups of $G$.  Taking preimages, we recover the theorem.
\end{proof}

It is natural to ask which values of $n$ in Theorem~\ref{thm:inversesemigroups} are possible, i.e.~what integers appear as the minimum index of a proper subgroup.

\begin{pro}\label{pro:minimumindexsubgroup}
Let $n\geq 2$.  There exists a group $G$ such that the minimum index of a proper subgroup of $G$ is $n$ if and only if $n\neq 4$.
\end{pro}
\begin{proof}
First, let $n$ be prime.  The only proper subgroup of $C_n$, the cyclic group of order $n$, is the trivial group, which has index $n$.  This implies every prime number, including two and three, can be found as the minimum index of proper subgroup.

Next, let $n\geq 5$ and consider $A_n$, the alternating group on $n$ points.  We see that $A_n$ has a subgroup of index $n$, namely $A_{n-1}$.  Assume to the contrary that $A_n$ has a proper subgroup $H$ of index $k<n$.  This implies that there is a homomorphism from $A_n$ into $S_k$, from the action of $A_n$ on the cosets of $H$.  This homomorphism is trivial since $A_n$ is simple, contradicting the fact that the induced action is transitive.  We conclude the minimum index of a proper subgroup of $A_n$ is $n$.  This shows that every $n\geq 5$ can be found as the minimum index of proper subgroup.

It remains to be shown that there are no groups where the minimal index of a proper subgroup is four.   Suppose that $G$ is a finite group with a subgroup $H$ of index 4.  This implies the existence of a homomorphism from $G$ into $S_4$, using the transitive action of $G$ on the cosets of $H$. Every transitive subgroup of $S_4$ has a subgroup of index 2 or 3, meaning that 4 is not the minimum index of a proper subgroup of $G$.   
\end{proof}

We conclude this section with a proof of Corollary~\ref{cor:allinversecoverings}, giving explicit examples of inverse semigroups belonging to each case in Theorem~\ref{thm:inversesemigroups}.

\begin{proof}[Proof of Corollary~\ref{cor:allinversecoverings}]
Let $n\geq 1$ and $X$ be an $n$-element set.  The symmetric inverse monoid $\cI_n$ is the semigroup of partial one-to-one functions from $X$ to $X$, i.e.~the set of partial functions that are injective on their domain.  We see $\cI_n$ is an inverse semigroup belonging to Case (i) of Theorem~\ref{thm:inversesemigroups} with $\sigma_i(\cI_n)=2$, since $\cI_n$ is the union of the set of bijections and the set of non-bijections.

Let the inverse subsemigroup of $\cI_n$ comprised of non-bijective elements be denoted by $S$.  We see that $S$ is generated by a single $\cJ$-class, namely the $\cJ$-class $J$ consisting of partial functions that are not defined on a single element.  This shows $S$ belongs to the third case of Theorem~\ref{thm:inversesemigroups}.  The principal factor, $J^*$, is isomorphic to $\cM^0[X,S_{n-1},X;I]$, where $S_{n-1}$ is the symmetric group on $n$ points.  Thus, provided $n\geq 3$, we see that $\sigma_i(S)=3$.

The groups $S_3$ and $A_4$ satisfy $\sigma_i(S_3)=\sigma_g(S_3)=4$ and $\sigma_i(A_4)=\sigma_g(A_4)=5$.

Now let $n\geq 5$. By Proposition~\ref{pro:minimumindexsubgroup}, the semigroup $S_n=\cM^0[\{1,2\},A_n,\{1,2\};I]$ belongs to subcase (a) of Theorem~\ref{thm:inversesemigroups} and satisfies $\sigma_i(S_n)=n+1$.
\end{proof}


\section{Covering Monoids}\label{sec:monoids}

In this section, we give a proof of Theorem~\ref{thm:monoids}, which deals with covering numbers of monoids with respect to subsemigroups, submonoids, and monoidal subsemigroups.  The following lemma describes the relationship between these covering numbers.  This result is clear since every submonoid of $M$ is also a monoidal subsemigroup of $M$ and every monoidal subsemigroup of $M$ is a subsemigroup of $M$.

\begin{lem}\label{lem:order_of_coverings}
Let $M$ be a monoid. Then $\sigma_s(M)\leq\sigma_m^*(M)\leq\sigma_m(M)$.
\end{lem}

Let $M$ be a monoid. Recall that $R_1$ and $L_1$ denote the $\cR$-classes and $\cL$-classes of $M$ containing $1$ respectively, see Definition~\ref{defn:greensrelations}.  Our proof of Theorem~\ref{thm:monoids} considers three cases for the monoid $M$: when the complement of $R_1$ in $M$ is empty, when $M-R_1$ is not empty and $R_1$ contains a non-identity element, and lastly, when $M-R_1$ is not empty but $R_1=\{1\}$.  The following three lemmas address these cases.

\begin{lem}\label{lem:monoidsaregroups}
If $M-R_1=\emptyset$, then $M$ is a group and $\sigma_m(M)=\sigma_m^*(M)=\sigma_s(M)$.
\end{lem}
\begin{proof}
Let $f\in R_1$.  Then there exists $g\in M$ such that $fg=1$.  Notice that $g\in R_1$ also, so there exists $h\in M$ such that $gh=1$.  We see that $f=fgh=h$ and therefore $g$ is a two-sided inverse of $f$.  This implies $M$ is a group.

Since $M$ is a group, $M$ contains a single idempotent.  Therefore any monoidal subsemigroup of $M$ must contain 1 and is thus a submonoid.  We conclude $\sigma_m(M)=\sigma_m^*(M).$

Suppose that $\sigma_s(M)=n$ for some integer $n\geq 2$.  Then there exists a set $\{S_1,\ldots,S_n\}$ of proper subsemigroups of $M$ such that $\bigcup S_i=M$.  We claim that $S_i\cup \{1\}$ is a proper submonoid of $M$.  It is clear that $S_i\cup \{1\}$ is a submonoid of $M$.  Suppose that $S_i\cup \{1\}$ is not a proper subset of $M$.  Therefore $S_i=M-\{1\}$. However, since $M$ is a group, $M-\{1\}$ is not a closed subset of $M$ (unless $M=\{1\}$, in which case $S_i=\emptyset$ and we achieve a contradiction).  This shows $S_i\cup \{1\}$ is a proper submonoid of $M$.  We conclude that $\{S_1\cup \{1\},\ldots,S_n\cup \{1\}\}$ is a set of proper submonoids whose union is $M$ and thus $\sigma_m(M)\leq \sigma_s(M)$. Using Lemma~\ref{lem:order_of_coverings}, we conclude $\sigma_s(M)=\sigma_m^*(M)=\sigma_m(M)$.
\end{proof}

\begin{lem}\label{lem:monoidnonemptyeverything}
If $M-R_1\neq\emptyset$ and $R_1\neq \{1\}$, then $\sigma_m(M)=\sigma_m^*(M)=\sigma_s(M)=2$.
\end{lem}
\begin{proof}
Let $f,g\in R_1$.  Then there exist $h_f,h_g\in M$ such that $fh_f=1$ and $gh_g=1$.  We see that $fgh_gh_f=1$, so $fg\in R_1$.  We conclude that $R_1$ is a submonoid of $M$.

Now let $f,g\in M-R_1$.  Note that $1\not\in fS^1$ and $1\not\in gS^1$ since $f$ and $g$ are not $\cR$-related to $1$.  Also, we can see that $fgS^1\subseteq fS^1$ and therefore $fgh\neq 1$ for all $h\in M$.  This shows $fg$ is not $\cR$-related to $1$ and $M-R_1$ is a semigroup.  Therefore $\sigma_s(M)=2$, since $M=R_1\cup (M-R_1)$.

Since $R_1\neq \{1\}$, the class $R_1$ contains a non-identity element.  This implies that\linebreak $(M-R_1)\cup \{1\}$ is a proper submonoid of $M$.  Therefore $M=R_1\cup((M-R_1)\cup\{1\})$ and $\sigma_m(M)=\sigma_m^*(M)=2$.
\end{proof}

\begin{lem}\label{lem:semigroupwithidentity}
If $M-R_1\neq\emptyset$ and $R_1=\{1\}$, then $\sigma_m^*(M)\leq\sigma_m(M)=\sigma_s(M-\{1\})$ and $\sigma_s(M)=2$.
\end{lem}
\begin{proof}
If $f\in L_1$, then $f$ has a left-inverse $g\in M$ such that $gf=1$.  However, this implies that $g\in R_1$, so $g=f=1$.  We conclude that $L_1=R_1=\{1\}$.

Let $S=M-\{1\}$.  Note that for all $f,g\in S$, we have $fg\neq 1$, since this would imply $f\in R_1$ and $g\in L_1$.  Therefore $S$ is a semigroup and $M=S\cup\{1\}$.  We immediately see that $\sigma_s(M)=2$.

Suppose that $\sigma_s(S)=n$ for some $n\in \mathbb{N}$.  Then there exists a set $\{S_1,\ldots, S_n\}$ of proper subsemigroups of $S$ such that $\bigcup S_i=S$.  Consider the set $\{S_1\cup \{1\},\ldots,S_n\cup\{1\}\}$ of proper submonoids of $M$.  We see that $\bigcup(S_i\cup\{1\})=M$, so $\sigma_m^*(M)\leq \sigma_m(M)\leq \sigma_s(S)$.

Now suppose that $\sigma_m(M)=n$ for some $n\in\mathbb{N}$.  Then there exists a set $\{M_1,\ldots, M_n\}$ of proper submonoids of $M$ such that $\cup M_i=M$.  Note that $1\in M_i$ for each $i$.  Consider the set $\{M_1-\{1\},\ldots,M_n-\{1\}\}$ of subsets of $S$.  It is clear that for each $i$, $M_i-\{1\}$ is a proper subsemigroup of $S$ and $\bigcup(M_i-\{1\})=S$, so $\sigma_s(S)\leq \sigma_m(M)$.  We conclude $\sigma_s(S)=\sigma_m(M)$.
\end{proof}

We combine the previous three lemmas to complete the proof of Theorem~\ref{thm:monoids}.

\begin{proof}[Proof of Theorem~\ref{thm:monoids}]
Let $M$ be a monoid.  First, if $M-R_1=\emptyset$, then $M$ is a group and $\sigma_m(M)=\sigma_m^*(M)=\sigma_s(M)$ by Lemma~\ref{lem:monoidsaregroups}.  Next, if $M-R_1\neq\emptyset$ and $R_1\neq \{1\}$, then $\sigma_m(M)=\sigma_m^*(M)=\sigma_s(M)=2$ by Lemma~\ref{lem:monoidnonemptyeverything}.  Lastly, when $M-R_1\neq\emptyset$ and $R_1=\{1\}$, then $M$ is a semigroup with an identity adjoined.  In this case, using Lemma~\ref{lem:semigroupwithidentity} and letting $S=M-\{1\}$, we have $\sigma_m(M)^*\leq\sigma_m(M)=\sigma_s(S)$ and $\sigma_s(M)=2$.
\end{proof}

As a corollary of Theorem~\ref{thm:monoids}, we see that the covering number of a monoid $M$ with respect to semigroups is always two, except possibly when $M$ is a group. Note that this applies to infinite monoids as well. 

\begin{cor}\label{cor:monoidsnotgroups}
Let $M$ be a monoid such that $M$ is not a group.  Then $\sigma_s(M)=2$.
\end{cor}

Combining Theorem~\ref{thm:finitesemigroups} and Theorem~\ref{thm:monoids} allows us to give the following more precise characterization of covering numbers of finite monoids.

\begin{cor}
Let $M$ be a finite monoid.
\begin{enumerate}[(i)]
\item If $M$ is a group, then $\sigma_m(M)=\sigma_m^*(M)=\sigma_s(M)=\sigma_g(M)$.
\item If $S=M-\{1\}$ is a group, then $\sigma_m^*(M)=2=\sigma_s(M)$ and $\sigma_m(M)=\sigma_g(S)$.
\item If $S=M-\{1\}$ is a monogenic semigroup that is not a group, then $\sigma_s(M)=2$ and $\sigma_m^*(M)=\sigma_m(M)=\infty$.
\item Otherwise, $\sigma_m(M)=\sigma_m^*(M)=\sigma_s(M)=2$
\end{enumerate}
\end{cor}

To conclude this section, we observe that in one particular case, the covering number of a monoid with repsect to submonoids and monoidal subsemigroups may differ.  We give a complete characterization of when this case occurs.

\begin{pro}
We have $\sigma_m^*(M)<\sigma_m(M)$ if and only if $M-\{1\}$ is a group and $\sigma_s(M-\{1\})>2$.  Also, $\sigma_m^*(M)<\sigma_m(M)$ implies $\sigma_m^*(M)=2$
\end{pro}
\begin{proof}
Suppose that $\sigma_m^*(M)<\sigma_m(M)$. By Lemmas~\ref{lem:monoidsaregroups}, \ref{lem:monoidnonemptyeverything}, and \ref{lem:semigroupwithidentity}, we see that we must have $M-R_1\neq\emptyset$ and $R_1=\{1\}$.  Also suppose that $\sigma_m^*(M)=n$ for some $n\in\mathbb{N}$.   Then there exists a set $\{M_1,\ldots,M_n\}$ of proper monoidal subsemigroups of $M$ such that $\bigcup M_i=M$.  Consider the set $\{M_1\cup\{1\},\ldots,M_n\cup\{1\}\}$ of subsets of $M$.   It is clear that $M_i\cup\{1\}$ is a submonoid of $M$ for each $i$, and $\bigcup(M_i\cup\{1\})=M$.  However, since $\sigma_m^*(M)<\sigma_m(M)$, there must exist an $i$ such that $M_i\cup \{1\}$ is not proper in $M$, i.e. $M_i=M-\{1\}$.  This shows that $M-\{1\}$ is a monoidal subsemigroup of $M$.  We now see that $\sigma_m^*(M)=2$, since $M=(M-\{1\})\cup\{1\}$.  We conclude $\sigma_s(M-\{1\})=\sigma_m(M)>2$ and that $M-\{1\}$ is a group, using Corollary~\ref{cor:monoidsnotgroups}.

The reverse direction is clear, because $\sigma_m^*(M)=2$ and $\sigma_m(M)>2$.
\end{proof}

As a remark, for finite monoids $M$, we have $\sigma_m^*(M)<\sigma_m(M)$ if and only if $M-\{1\}$ is a non-monogenic group.  However, some infinite groups have covering number with respect to semigroups equal to two, such as $\mathbb{Z}$ under addition, so the requirement that $\sigma_s(M-\{1\})>2$ in the previous lemma is necessary, even when $M-\{1\}$ is a group.


\section{Open Questions}\label{sec:questions}
Although we have given a complete characterization of covering numbers of finite semigroups and finite inverse subsemigroups, the infinite case is largely unsolved.  Some methods in this paper can be extended to the infinite case, with obvious complications.  For instance, an infinite semigroup needs not have a maximal $\cJ$-class, preventing the use of Lemma~\ref{lem:SwithoutJ}.  Also, infinite simple and 0-simple semigroups may not be completely simple or completely 0-simple, so Rees's Theorem has limited usefulness.\\
\\
\noindent\textbf{Question 1}\quad\emph{What is $\sigma_s(S)$ for an infinite semigroup $S$?}
\\
\\
\noindent\textbf{Question 2}\quad\emph{What is $\sigma_i(I)$ for an infinite inverse semigroup $I$?}
\\

The covering number of groups with respect to semigroups is addressed in \cite{Donoven}, where the first author characterizes which groups have semigroup covering number equal to two (such as $\mathbb{Z}$).  Specifically, it is shown that for a group $G$, we have $\sigma_s(G)=2$ if and only if $G$ has a non-trivial left-orderable quotient.  Since covering numbers with respect to semigroups and groups are equivalent for finite groups, for every $n$ that is a group covering number, there exists a semigroup $S$ such that $\sigma_s(S)=n$.  However, not every integer is a group covering number, for instance 7 and 11, as mentioned in the introduction and discussed in \cite{GaronziKappeSwartz18}. This leads to the following question.
\\
\\
\noindent\textbf{Question 3}\quad\emph{Does there exist a semigroup $S$ such that $\sigma_s(S)=7$ or any other integer greater than two that is not a group covering number?}

\section*{Acknowledgements}

The authors would like to thank Arturo Magidin and our colleagues at Binghamton University for their helpful advice and feedback, especially with the proof of Proposition~\ref{pro:minimumindexsubgroup}.

\bibliography{references_mastercopy}

\end{document}